\documentclass[12pt]{amsart}
\usepackage{amssymb,amsmath,mathrsfs}
\usepackage[colorlinks=true,urlcolor=blue,
citecolor=red,linkcolor=blue,linktocpage,pdfpagelabels,
bookmarksnumbered,bookmarksopen]{hyperref}
\usepackage[active]{srcltx}
\usepackage{verbatim}
\usepackage{epsfig,graphicx,color,mathrsfs}
\usepackage{graphicx}
\usepackage{amsmath,amssymb,amsthm,amsfonts}
\usepackage{amssymb}
\usepackage[english]{babel}
\usepackage[hyperpageref]{backref}
\usepackage[left=2.7cm,right=2.7cm,top=3cm,bottom=3cm]{geometry}

\newcommand{\R}{{\mathbb R}}

\def \d {\mathrm{d}}
\def \Cp {\mathrm{Cap}}
\def \ep {\varepsilon}
\newtheorem{thm}{Theorem}[section]

\newtheorem{lem}[thm]{Lemma}

\newtheorem{rem}[thm]{Remark}
\newtheorem{defn}[thm]{Definition}

\numberwithin{equation}{section}

\title[]{Symmetry of intrinsically singular solutions of double phase problems}

\author[S.\,Biagi]{Stefano Biagi}
\author[F.\,Esposito]{Francesco Esposito}
\author[E.\,Vecchi]{Eugenio Vecchi}

\address[S.\,Biagi]{Politecnico di Milano - Dipartimento di Matematica
\newline\indent
Via Bonardi 9, 20133 Milano, Italy}
\email{stefano.biagi@polimi.it}

\address[F.\,Esposito]{Dipartimento di Matematica e Informatica, Universit\`a della Calabria
\newline\indent
Ponte Pietro Bucci 31B, 87036 Arcavacata di Rende, Cosenza, Italy}
\email{esposito@mat.unical.it}

\address[E.\,Vecchi]{Dipartimento di Matematica, Universit\`a di Bologna
\newline\indent
Piazza di Porta San Donato 5, 40126 Bologna, Italy}
\email{eugenio.vecchi2@unibo.it}



\thanks{The authors are members of INdAM-GNAMPA. 
F. Esposito is partially supported by PRIN project 2017JPCAPN (Italy): 
{\em Qualitative and quantitative aspects of nonlinear PDEs.}
}

\begin{document}
\begin{abstract}
We continue the study of positive singular solutions of PDEs arising from double phase functionals started in \cite{BEV}. In particular, we consider the 
case $p<q < 2$, and we relax the assumption on the capacity of the 
singular set using an intrinsic notion of capacity.

\end{abstract}

\maketitle

\medskip

\section{Introduction}\label{intro}
 Double phase functionals are integral functionals of the form 
 \begin{equation}\label{DoublePhase}
 u \mapsto \int_{\Omega}(|\nabla u|^p + a(x) |\nabla u|^q)\,\d x,
 \end{equation}
 where $\Omega \subset \mathbb{R}^N$, $1<p<q<N$ and $a(\cdot)\geq 0$. This class of functionals naturally appear in homogenization theory and in the study of strongly anisotropic materials (see, e.g., \cite{zhikov1}), and falls into the framework of the so called functionals with {\it non-standard growth} introduced by Marcellini 
 \cite{marcellini1,marcellini2}. The literature concerning functionals like \eqref{DoublePhase} is pretty vast and concerns as a main topic the regularity of minimizers, see e.g. \cite{BarColMin, ColMin1, ColMin2,ELM} and the references therein. 
 \medskip

In this paper we continue the study started in \cite{BEV} concerning qualitative properties of positive (and nontrivial) solutions $u\in C^{1}(\overline{\Omega}\setminus \Gamma)$ of a class of nonlinear PDEs which are closely related to the functional \eqref{DoublePhase}, namely
\begin{equation}\label{eq:Riey2}
  \left\{\begin{array}{rl}
  -\mathrm{div} \left(p|\nabla u|^{p-2}\nabla u 
  + q a(x)|\nabla u|^{q-2}\nabla u \right) = f(u) & \text{in $\Omega \setminus \Gamma$},\\
  u>0  & \text{in $\Omega \setminus \Gamma$},\\
  u = 0 & \textrm{on } \partial \Omega.
  \end{array}\right.
 \end{equation}
Here $\Omega \subset \R^N$ is a bounded and smooth domain, with $N \geq 2$ and $1< p < q < N$, while $\Gamma \subset \Omega$ is a closed set which is small in a proper (Sobolev)-capacitary sense. Under suitable assumptions on $a$ and $f$ (see below), and {\em only for $2\leq p<q<N$} we proved in \cite{BEV} that  {\em solutions} inherit symmetry from $\Omega$. Our aim is twofold:
\begin{itemize}
\item consider the case $1<p<q<2$;
\item relax the assumptions on $\Gamma$ using a more intrinsic notion of capacity.
\end{itemize}
We now clarify once for all what is the notion of solution to \eqref{eq:Riey2} we will work with.

\begin{defn} \label{def:solution} 
 We say that a function
 $u\in C^{1}({\overline\Omega}\setminus \Gamma)$ is a solution to 
 problem \eqref{eq:Riey2} if 
 it satisfies the following two properties:
 \begin{enumerate}
  \item $u > 0$ in $\Omega\setminus\Gamma$ and $u=0$ on $\partial \Omega$;
  \item for every $\varphi\in
	C^1_c(\Omega\setminus\Gamma)$ one has
	\begin{equation} \label{debil1}
	\int_\Omega \big(p|\nabla u|^{p-2}+qa(x)|\nabla u|^{q-2}\big)
	\langle \nabla u, \nabla
	\varphi\rangle\,\d x\, = \,\int_\Omega f(u)\varphi\,\d x,
	\end{equation}
 \end{enumerate}	
 \noindent where $\langle \cdot, \cdot \rangle$ denotes the standard scalar product in $\mathbb{R}^N$.
\end{defn}
 For results concerning the existence of solutions to \eqref{eq:Riey2} we refer  for example to \cite{CreGasWin,PapaWin}. We are finally ready to state our main result.
 \begin{thm}\label{thm:symmetry}
  Let $\tfrac{2N}{N+2}< p <q< 2$  and let $\Omega\subseteq\R^N$ be a convex open set, symmetric 
  wrt the  $x_1$-direction. Moreover, let $\Gamma\subseteq\Omega\cap\{x_1 = 0\}$
  be a closed set such that
  \begin{equation} \label{eq:assCappqzero}
   \mathrm{Cap}_{p,q}(\Gamma) = 0.
  \end{equation}
  Finally, we require $a$ and $f$ to satisfy the following assumptions:
  \begin{enumerate}
   \item $a\in L^\infty(\Omega)\cap C^1(\Omega)$ is \emph{non-negative and independent of $x_1$};
   
   \item $f:\R \rightarrow \R$ is a locally Lipschitz continuous function with $f(s)>0$ for $s>0$.
  \end{enumerate}  
  Then, any solution $u\in C^1(\overline{\Omega}\setminus\Gamma)$ 
  to \eqref{eq:Riey2}
  is symmetric wrt the hyperplane $\{x_1=0\}$ and increasing 
  in the $x_1$-direction in $\Omega \cap \{x_1<0\}$. 
\end{thm}
\begin{rem} \label{rem:casogenerale}
We explicitly stress that, even if Theorem \ref{thm:symmetry} is stated
 under the assumption $p<q<2$, by combining the techniques exploited
 in this paper with the approach carried out in \cite{BEV} one can easily
 see that the same result holds when
 $$\frac{2N}{N+2}<p < 2 \leq q.$$
 We refer to Remark \ref{rem:casogeneraleDett} for some further detail.
\end{rem}
The result is of classical flavor and resembles the seminal papers \cite{GNN} and \cite{BN} where the authors proved symmetry of solutions of semilinear elliptic equations by means of the moving plane method introduced by Alexandrov \cite{A} and Serrin \cite{serrin}.
The power of this very elegant method is witnessed by the huge existing literature: 
see \cite{Dan, Troy} for the case of cooperative elliptic systems, 
see, e.g., \cite{Da2, DP, DS1, FMRS, FMS2}
for the case of quasilinear equations and \cite{Dan2} for the case
of non-smooth domains.
\medskip

The present situation is slightly different in we are considering singular solutions. In order to face the natural technical difficulties which arise in this case, we exploit a rather recent modification of 
the moving plane method introduced by Sciunzi in \cite{Dino} in the case of singular solutions. This technique has already proved to be flexible enough to  be adapted in the case of unbounded sets \cite{EFS}, $p$-Laplacian operator \cite{MPS}, cooperative elliptic systems \cite{BVV,esposito}, fractional Laplacian \cite{MPS}, mixed local--nonlocal elliptic operators \cite{BDVV}, cooperative quasilinear systems 
\cite{BEMV}, double phase problems with $p\geq 2$ \cite{BEV}.
Concerning qualitative properties
for \emph{singular solutions} of certain PDEs, it is worth highlighting that
there are several works
dealing with \emph{point-type singularities}, see, e.g., \cite{BVV0, CLN2, Ter}. 
As a matter of fact, since our assumption 
\eqref{eq:assCappqzero} is clearly satisfied by
this kind of singularities, the above
Theorem \ref{thm:symmetry} can be seen as a generalization of \cite{CLN2,Ter} 
to the double phase setting.
\medskip

Let us now spend a few comments on Theorem \ref{thm:symmetry}. As already mentioned, we use a variant of the moving plane method which needs
few classical tools, like 
 a weak/strong com\-pa\-ri\-son principle and Hopf-type Lemma; 
 owing to \cite{riey} and \cite{BEV}, these tools are all at our disposal.
 In order to better explain the delicate technical difference with respect to the case considered in \cite{BEV}, let us recall that already in the pure $p$-Laplacian case (i.e., $a(x)\equiv 0$) when $1<p<2$, the operator may present a pretty singular behavior near the set $\Gamma$ and hence the inverse of the weight $\rho:=|\nabla u|^{p-2}$ may not have the right summability properties which allow to use a weighted Sobolev inequality of Trudinger \cite{Tru}. 
 
 In \cite{EMS}, the authors avoided this problem (for the pure $p$-Laplacian case)
 by performing an 
accurate analysis of the behavior of the gradient of the solution near the set $\Gamma$ based on previous results contained in \cite{PQS}. This is currently out of reach for us. Nevertheless, in the recent \cite{BEMV} a similar problem appeared in the case of quasilinear systems with a gradient term and we have been able to circumvent it by a nice use of the classical H\"{o}lder inequality. This trick can be adapted to the present setting as well. We believe that the simplicity of the argument is kind of surprising, but we have to pay the price of a not optimal lower bound for $p$ (i.e. $p>\tfrac{2N}{N+2}$), which we cannot push down to $1$. 
\vspace*{0.05cm}

A second novelty with respect to \cite{BEV} concerns the use of an intrinsic capacity (denoted by $\mathrm{Cap}_{p,q}(\cdot)$) recently used in \cite{DFM}. First of all, this provides an immediate generalization of our previous result because
\begin{equation*}
\mathrm{Cap}_q (E)=0 \Rightarrow \mathrm{Cap}_{p,q}(E)=0.
\end{equation*}
Moreover, we want to stress that this choice has also a technical impact in the proof of the crucial Lemma \ref{leaiuto}. Incidentally, we also point out that the aforementioned Lemma reads exactly as its counterpart \cite[Lemma 2.4]{BEV}, but the proof has to be 
performed once again because we are now considering the case $p<q<2$.
\medskip
 
 The plan of the paper is the following:

\begin{itemize}
	\item[-] In Section \ref{sec.preliminaries} we fix the notations used 
	throughout the paper. Moreover, we recall the notion of $(p,q)$-capacity and some related theorems. Finally, we prove the key Lemma \ref{leaiuto}.
	
	\item[-] In Section \ref{sec.mainresult1} we prove Theorem \ref{thm:symmetry}.

\end{itemize}

\section{Notations and auxiliary results} \label{sec.preliminaries}
 The aim of this section is twofold: on the one hand, we fix once and for all
 the relevant notation used throughout the paper; on the other hand,
 we present some auxiliary results which shall be key ingredients
 for the proof of Theorem \ref{thm:symmetry}.
 
 \subsection{A review of capacities.} 
 Let $1\leq r \leq N$ be fixed, and let $K\subseteq\R^N$ be a \emph{compact set}.
 We remind that the classical Sobolev $r$-capacity of $K$ is defined as
 \begin{equation*}
 \mathrm{Cap}_r(K):=
  \inf \bigg\{\int_{\mathbb{R}^N}|\nabla \varphi|^r\,\d x\,:\,
  \text{$\varphi \in C^{\infty}_c(\R^N)$ and $\varphi\geq 1$ on $K$}\bigg\}.
 \end{equation*}
 Moreover, if $D\subseteq\R^N$ is any bounded set containing $K$,
 it is possible to define the $r$-capacity of \emph{the condenser
 $(K,D)$} in the following way
 \begin{equation}\label{eq:q-capacityrel}
 \mathrm{Cap}^D_r(K) :=
  \inf \bigg\{\int_{\mathbb{R}^N}|\nabla \varphi|^r\,\d x\,:\,
  \text{$\varphi \in C^{\infty}_c(D)$ and $\varphi\geq 1$ on $K$}\bigg\}.
 \end{equation}
 More recently there has been a certain interest in the study of capacities connected with Orlicz-Sobolev spaces, see e.g., \cite{BHH}. In particular, when dealing with the double-phase fun\-ctio\-nals, one is naturally led to consider a so-called $(p,q)$-capacity, see \cite{DFM}: 
 if {$K\subset \Omega$} is a compact set, we define the \emph{$(p,q)$-capacity of $K$} as
 \begin{equation*}
 \mathrm{Cap}_{p,q}(K):=
  \inf \bigg\{\int_{\R^n}(|\nabla \varphi|^p + a(x)|\nabla \varphi|^q)\,\d x\,:\,
  \text{$\varphi \in C^{\infty}_c(\Omega)$ and $\varphi\geq 1$ on $K$}\bigg\}.
 \end{equation*} 

 Similarly to \eqref{eq:q-capacityrel}, it is then possible to define the $(p,q)$-capacity of the 
 condenser $(K,D)$ (where $D\subseteq\Omega$ is a bounded open set containing $K$);
 with this definition at hand, we say that $K$ has \emph{vanishing $(p,q)$-capacity},
 and we write $\mathrm{Cap}_{p,q}(K) = 0$, if
 $$\mathrm{Cap}^D_r(K) = 0\qquad\text{for every open set $D\supseteq K$}.$$
 \begin{rem} \label{rem:capacityprop}
  We list, for a future reference, a couple of remarks
  highlighting the relation between the $(p,q)$-capacity
  and the classical Sobolev capacities $\mathrm{Cap}_p(\cdot),\,\mathrm{Cap}_q(\cdot)$.
  In what follows, we tacitly understand that $K\subseteq\Omega$ is a fixed
  compact set.
  \begin{enumerate}
   \item If $K$ has \emph{vanishing $q$-capacity}, then $\mathrm{Cap}_{p,q}(K) = 0$.
   To prove this fact it suffices to observe that,
   if $D\subseteq\Omega$ is a bounded open set containing $K$ and if
   $\varphi\in C_c^\infty(D)$ sa\-ti\-sfi\-es $\varphi\geq 1$ on $K$, 
   by the boundedness of $a(\cdot)$
   we have
   \begin{align*}
    &\int_{\mathbb{R}^N}(|\nabla \varphi|^p + a(x)|\nabla \varphi|^q)\,\d x \\
    & \qquad \leq
    |\Omega|^{1-\frac{p}{q}}\bigg(\int_{\R^n}|\nabla\varphi|^q\,\d x\bigg)^{\frac{p}{q}}
    +
    \|a\|_{L^\infty(\Omega)}\int_{\R^N}|\nabla\varphi|^q\,\d x,
   \end{align*}
   where $|\cdot|$ denotes the classical Lebesgue measure in $\R^N$.
   \vspace*{0.05cm}
   
   \item If $K$ has vanishing $(p,q)$-capacity, then $\mathrm{Cap}_p(K) = 0$.
   To prove this fact it suffices to observe that, if $D\subseteq\Omega$ is a bounded 
   open set containing $K$ and if
   $\varphi\in C_c^\infty(D)$ sa\-ti\-sfi\-es $\varphi\geq 1$ on $K$, 
   since $a(\cdot)\geq 0$ in $\Omega$ we have
   $$\int_{\mathbb{R}^N}(|\nabla \varphi|^p + a(x)|\nabla \varphi|^q)\,\d x\geq 
   \int_{\R^N}|\nabla \varphi|^p\,\d x.$$
   In particular, if $\mathrm{Cap}_{p,q}(K) = 0$ we deduce that $|K| = 0$.
  \end{enumerate}
 \end{rem}
 \subsection{Notations for the moving plane method.}
 Let $\Gamma\subseteq\Omega\subseteq\R^N$ be as in the statement
 of Theorem \ref{thm:symmetry}, and let
 $u\in C^1(\overline{\Omega}\setminus\Gamma)$ be a solution
 of \eqref{eq:Riey2}. For any fixed $\lambda\in\R$, we in\-di\-ca\-te by $R_\lambda$
 the reflection trough the hy\-per\-pla\-ne $\Pi_\lambda := \{x_1 = \lambda\}$,
 that is,
 \begin{equation*}
  R_\lambda(x) = x_\lambda := (2\lambda-x_1,x_2,\ldots,x_N)
  \qquad (\text{for all $x\in\R^N$});
 \end{equation*}
 accordingly, we define the function
 \begin{equation} \label{eq.defulambda}
  u_\lambda(x) := u(x_\lambda), \qquad\text{for all
  $x\in R_\lambda\big(\overline{\Omega}\setminus\Gamma\big)$}.
 \end{equation}
 We point out that, since $u$ solves \eqref{eq:Riey2}
 and $a$ is \emph{independent of $x_1$}, one has
 \begin{enumerate}
  \item  
 $u_\lambda\in C^1(R_\lambda(\overline{\Omega}\setminus\Gamma))$;
  \item $u_\lambda > 0$ in $R_\lambda(\Omega\setminus\Gamma)$ and 
  $u_\lambda \equiv 0$ on $R_\lambda(\partial\Omega\setminus\Gamma)$;
  \item for every test function $\varphi\in C^1_c(R_\lambda(\Omega\setminus\Gamma))$
  one has
  \begin{equation} \label{eq.PDEulambda}
  \int_{R_\lambda(\Omega)}\big( p|\nabla u_\lambda|^{p-2}+qa(x)|\nabla u_\lambda|^{q-2}\big)
  \langle \nabla u_\lambda,\nabla \varphi\rangle\,\d x\,=
  \,\int_{R_\lambda(\Omega)}f(u_\lambda)\varphi\,\d x.
  \end{equation}
 \end{enumerate}
 To proceed further, we let
 \begin{equation*}
  \mathbf{a} = \mathbf{a}_\Omega := \inf_{x\in\Omega}x_1
 \end{equation*}
 and we observe that, since $\Omega$ is (bounded and) symmetric with respect to
 the $x_1$-direction,
 we certainly have $-\infty < \mathbf{a} < 0$.
 Hence, for every $\lambda\in(\mathbf{a},0)$ we can set
 \begin{equation*}
 \Omega_\lambda := \{x\in\Omega:\,x_1<\lambda\}.
 \end{equation*}
 Notice that the convexity of
 $\Omega$ in the $x_1$-direction ensures that
 \begin{equation} \label{eq.inclusionOmegalambda}
  \Omega_\lambda\subseteq R_\lambda(\Omega)\cap \Omega.
 \end{equation}
 Finally, for every $\lambda\in(\mathbf{a},0)$ we define the function
 \begin{equation} \label{eq:defwlambda}
 w_\lambda(x) := (u-u_\lambda)(x), \qquad \text{for
 $x\in (\overline{\Omega}\setminus\Gamma)\cap R_\lambda(\overline{\Omega}\setminus\Gamma)$}.
 \end{equation}
 On account of \eqref{eq.inclusionOmegalambda}, 
 $w_\lambda$ is surely well-posed on
 $\overline{\Omega}_\lambda\setminus R_\lambda(\Gamma)$.
 
 \subsection{Auxiliary results.}
 From now on, we assume that all the hypotheses of Theorem \ref{thm:symmetry}
 are satisfied. Moreover, we tacitly inherit all the notation
 introduced so far.
 \medskip
 
 To begin with, we remind some
 identities between vectors in $\R^N$ which
 are very useful in dealing
 with quasilinear operators:
 \emph{for every $s\in (1,2)$ 
 there exist constants $C_1,C_2 > 0$, only depending on $s$, such that,
 for every $\eta,\eta'\in\R^N$, one has}
 \begin{equation}\label{eq:inequalities}
  \begin{split}
 & \langle |\eta|^{s-2}\eta-|\eta'|^{s-2}\eta', \eta- \eta'\rangle \geq C_1
(|\eta|+|\eta'|)^{s-2}|\eta-\eta'|^2, \\[0,15cm]
& \big| |\eta|^{s-2}\eta-|\eta'|^{s-2}\eta '| \leq C_2|\eta-\eta'|^{s-1}.
\end{split}
\end{equation}
We refer, e.g., to \cite{Da2} for a proof of 
\eqref{eq:inequalities}. \medskip

 Next, we need to define
 an \emph{ad-hoc} family of Sobolev functions in $\Omega$ allowing
 us to `cut off' of the singular set $\Gamma$. To this end, let
 $\varepsilon >0$ be small enough such that
 $$\mathcal{B}^{\lambda}_{\varepsilon}
 := \big\{x\in\R^N:\,\mathrm{dist}(x, R_{\lambda}(\Gamma)) < \varepsilon\big\}
 \subseteq\Omega.$$ 
 As for the classical capacity, being $R_\lambda$ an affine map, we see that
 $$\Cp_{p,q}\big(R_\lambda(\Gamma)\big) = 0.$$
 then, by definition, there exists $\varphi_{\varepsilon}\in 
 C_c^\infty(\mathcal{B}^{\lambda}_{\varepsilon})$
 such that
 \begin{equation} \label{eq.choicephieps}
  \text{$\varphi_\varepsilon\geq 1$ on $R_\lambda(\Gamma)$}\qquad
   \text{and}\qquad
   \int_{\mathcal{B}^{\lambda}_{\varepsilon}} 
   (|\nabla
  \varphi_{\varepsilon}|^p + a(x)|\nabla
  \varphi_{\varepsilon}|^q)\,\d x < \varepsilon.
 \end{equation}
 We then consider the Lipschitz functions
 \begin{itemize}
  \item $T(s) := \max\{0;\min\{s;1\}\}$ (for $s\in\R$), \vspace*{0.03cm}
  \item $g(t) := \max\{0;-2s+1\}$ (for $t\geq 0$)
 \end{itemize}
 and we define, for $x\in\R^N$,
 \begin{equation*}
  \psi_{\varepsilon}(x):= g\big(T(\varphi_{\varepsilon}(x))\big).
 \end{equation*}
 In view of \eqref{eq.choicephieps}, and taking into account
 the very definitions of $T$ and $g$, it is not difficult to recognize that
 $\psi_\varepsilon$ satisfy the following properties:
 \begin{enumerate}
  \item $\psi_\varepsilon \equiv 1$ on $\R^N\setminus \mathcal{B}^\lambda_\varepsilon$
  and $\psi_\varepsilon\equiv 0$ on some neighborhood of $R_\lambda(\Gamma)$, say
  $\mathcal{V}^\lambda_\ep\subseteq \mathcal{B}^{\lambda}_{\varepsilon}$; 
  \vspace*{0.03cm}
  
  \item $0\leq \psi_\varepsilon\leq 1$ on $\R^N$;
  \vspace*{0.03cm}
  
  \item $\psi_\varepsilon$ is Lipschitz-continuous in $\R^N$, so that
  $\psi_\varepsilon\in W^{1,\infty}(\R^N)$;
  \vspace*{0.03cm}
  
  \item there exists a constant $C > 0$, independent of $\varepsilon$, such that
  \begin{equation*}
   \int_{\R^N}(|\nabla
  \psi_{\varepsilon}|^p + a(x)|\nabla
  \psi_{\varepsilon}|^q)\,\d x \leq C \varepsilon.
  \end{equation*}
 \end{enumerate}
 \medskip
 
 With the family $\{\psi_\varepsilon\}_\varepsilon$ at hand, we can prove the following
 key lemma.
 \begin{lem}\label{leaiuto}
 Let $1<p<q<2$.
   For any fixed $\lambda\in(\mathbf{a},0)$ we have
	\begin{equation*}
	\int_{\Omega_\lambda} \big(p(|\nabla u| + |\nabla u_\lambda|)^{p-2}+
	qa(x)(|\nabla u| + |\nabla u_\lambda|)^{q-2}\big)\cdot|\nabla
	w_\lambda^+|^2\,\d x\leq
	\mathbf{c}_0,
	\end{equation*}
	where $\mathbf{c}_0 > 0$ is a constant only depending on $p,q,\lambda$ and
	$\|u\|_{L^\infty(\Omega_\lambda)}$.
 \end{lem}
 \begin{proof}
  For every fixed $\varepsilon > 0$, we consider the function
  $$\varphi_\varepsilon(x):= 
  \begin{cases}
   w_\lambda^+(x)\,\psi_\varepsilon^{p+q}(x) = (u-u_\lambda)^+(x)\,\psi_\varepsilon^{p+q}(x),
   & \text{if $x\in\Omega_\lambda$}, \\
   0, & \text{otherwise}.
   \end{cases}$$
  We claim that the following assertions hold:
  \begin{itemize}
   \item[(i)] $\varphi_\ep\in \mathrm{Lip}(\R^N)$;
   \item[(ii)]$\mathrm{supp}(\varphi_\ep)\subseteq\Omega_\lambda$ and
  $\varphi_\ep\equiv 0$ near $R_\lambda(\Gamma)$.
  \end{itemize}
  In fact, since $u\in C^1(\overline{\Omega}_\lambda)$ and $u_\lambda\in C^1(
  \overline{\Omega}_\lambda\setminus R_\lambda(\Omega))$,
  we have 
  $w_\lambda^+\in \mathrm{Lip}(\overline{\Omega}_\lambda\setminus V)$
  \emph{for every o\-pen set $V\supseteq R_\lambda(\Gamma)$};
  as a consequence,
  reminding that $\psi_\varepsilon\in\mathrm{Lip}(\R^N)$
  and $\psi_\varepsilon\equiv 0$ on a ne\-igh\-bor\-hood of $R_\lambda(\Gamma)$,
  we get
  $\varphi_\varepsilon\in \mathrm{Lip}(\overline{\Omega}_\lambda)$.
  On the other hand, since 
  $\varphi_\varepsilon\equiv 0$ on $\partial\Omega_\lambda$, we easily conclude that
  $\varphi_\ep\in \mathrm{Lip}(\R^N)$,
  as claimed. As for assertion (ii), it is a direct consequence
  of the very definition of $\varphi_\ep$ and of the fact that
  $$\text{$\psi_\ep\equiv 0$ on $\mathcal{V}^\lambda_\ep\supseteq R_\lambda(\Gamma)$}.$$
  On account of properties (i)-(ii) of $\varphi_\ep$,
   a standard density argument allows us to use
  $\varphi_\ep$ as a test function \emph{both in \eqref{debil1} and
  \eqref{eq.PDEulambda}}; reminding that $a$ is independent of $x_1$, this gives
	\begin{equation*}
	\begin{split} 
     &
     p\int_{\Omega_\lambda}
     \langle |\nabla u|^{p-2} \nabla u - |\nabla u_\lambda|^{p-2}\nabla u_\lambda,\nabla
     \varphi_\ep\rangle\,\d x
     \\
     & \qquad\qquad+ q
     \int_{\Omega_\lambda}a(x)\,\langle
     |\nabla u|^{q-2} \nabla u - |\nabla u_\lambda|^{q-2}\nabla u_\lambda,
		\nabla \varphi_\ep\rangle\,\d x \\
		&\quad = \int_{\Omega_\lambda} (f(u)-f(u_\lambda))\varphi_\ep\,\d x.
	\end{split}
  \end{equation*}
  By unraveling the very definition of $\varphi_\ep$, we then obtain
  \begin{equation} \label{eq.identityIntermediate}
	\begin{split} 
     &
     p\int_{\Omega_\lambda}
     \psi_\ep^{p+q}\cdot
     \langle |\nabla u|^{p-2} \nabla u - |\nabla u_\lambda|^{p-2}\nabla u_\lambda,\nabla
     w_\lambda^+\rangle\,\d x
     \\
     & \qquad\quad+ q
     \int_{\Omega_\lambda}\psi_\ep^{p+q}\cdot a(x)\,\langle
     |\nabla u|^{q-2} \nabla u - |\nabla u_\lambda|^{q-2}\nabla u_\lambda,
		\nabla w_\lambda^+\rangle\,\d x \\
    & \qquad\quad
    + p(p+q)
    \int_{\Omega_\lambda}
     w_\lambda^+\cdot\psi_\ep^{p+q-1}\cdot
     \langle |\nabla u|^{p-2} \nabla u - |\nabla u_\lambda|^{p-2}\nabla u_\lambda,\nabla
     \psi_\ep\rangle\,\d x \\
     & \qquad\quad
    + q(p+q)
    \int_{\Omega_\lambda}w_\lambda^+\cdot\psi_\ep^{p+q-1}\cdot a(x)\,\langle
     |\nabla u|^{q-2} \nabla u - |\nabla u_\lambda|^{q-2}\nabla u_\lambda,\nabla
     \psi_\ep\rangle\,\d x \\
		& \quad = \int_{\Omega_\lambda} (f(u)-f(u_\lambda))\,w_\lambda^+\,\psi_\ep^{p+q}\,\d x.
	\end{split}
  \end{equation}
  We now observe that the integral in the
  right-hand side of \eqref{eq.identityIntermediate} is actually
  performed on the set
  $\mathcal{O}_\lambda := \{x\in\Omega_\lambda:\,u\geq u_\lambda\}\setminus R_\lambda(\Gamma)$;
  moreover, for every $x\in\mathcal{O}_\lambda$ we have
  $$0\leq u_\lambda(x)\leq u(x)\leq \|u\|_{L^\infty(\Omega_\lambda)}.$$
  As a consequence, since $f$ is locally Lipschitz-continuous on $\R$, we have
  \begin{equation} \label{eq.estimfLipschitz}
   \begin{split}
    \int_{\Omega_\lambda} (f(u)-f(u_\lambda))\,w_\lambda^+\,\psi_\ep^{p+q}\,\d x
   &= \int_{\Omega_\lambda} \frac{f(u)-f(u_\lambda)}{u-u_\lambda}\,(w_\lambda^+)^2\,\psi_\ep^{p+q}\,\d x
   \\
   & \leq L\int_{\Omega_\lambda}
   (w_\lambda^+)^2\,\psi_\ep^{p+q}\,\d x,
   \end{split}
  \end{equation}
  where $L = L(f,u,\lambda)> 0$ 
  is the Lipschitz constant of $f$ on 
  the interval $[0,\|u\|_{L^\infty(\Omega_\lambda)}]\subseteq\R$.
  Starting from \eqref{eq.identityIntermediate} and
  using both \eqref{eq.estimfLipschitz} and the estimates in \eqref{eq:inequalities}, we then get
  \begin{equation} \label{eq.tostartfrom}
   \begin{split}
   & C_1\int_{\Omega_\lambda}\psi_\ep^{p+q}\,\big\{
   p(|\nabla u|+|\nabla u_\lambda|)^{p-2}
   + q a(x)\,(|\nabla u|+|\nabla u_\lambda|)^{q-2}\big\}
   \cdot |\nabla w_\lambda^+|^2\,\d x \\
   & \qquad 
   \leq p\int_{\Omega_\lambda}
     \psi_\ep^{p+q}\,
     \langle |\nabla u|^{p-2} \nabla u - |\nabla u_\lambda|^{p-2}\nabla u_\lambda,\nabla
     w_\lambda^+\rangle\,\d x
     \\
     & \qquad\quad+ q
     \int_{\Omega_\lambda}\psi_\ep^{p+q}\, a(x)\,\langle
     |\nabla u|^{q-2} \nabla u - |\nabla u_\lambda|^{q-2}\nabla u_\lambda,
		\nabla w_\lambda^+\rangle\,\d x \\
    & \qquad \leq p(p+q)
    \int_{\Omega_\lambda}
     w_\lambda^+\,\psi_\ep^{p+q-1}\cdot
     \big| |\nabla u|^{p-2} \nabla u - |\nabla u_\lambda|^{p-2}\nabla u_\lambda\big|\,
     |\nabla
     \psi_\ep|\,\d x \\
     & \qquad\quad
    + q(p+q)
    \int_{\Omega_\lambda}w_\lambda^+\,\psi_\ep^{p+q-1}\cdot a(x)\,\big|
     |\nabla u|^{q-2} \nabla u - |\nabla u_\lambda|^{q-2}\nabla u_\lambda
     \big|\,|\nabla \psi_\ep|\,\d x \\
     & \qquad\quad 
     +L_f\,\int_{\Omega_\lambda}
   (w_\lambda^+)^2\,\psi_\ep^{p+q}\,\d x \\
   & \qquad
   \leq C_2 \, p(p+q) \int_{\Omega_{\lambda}}|\nabla w_{\lambda}^+|^{p-1}\psi_\ep^{p+q-1}|\nabla \psi_\ep| w_{\lambda}^{+}\, \d x \\
   &\qquad \quad + C_2 \, q(p+q) \int_{\Omega_{\lambda}}a(x)|\nabla w_{\lambda}^+|^{q-1}\psi_\ep^{p+q-1}|\nabla \psi_\ep| w_{\lambda}^{+}\, \d x \\
   & \qquad\quad 
     +L_f\,\int_{\Omega_\lambda}
   (w_\lambda^+)^2\,\psi_\ep^{p+q}\,\d x \\
   & \qquad \leq C_0\,\bigg(I_p+I_q+\int_{\Omega_\lambda}
   (w_\lambda^+)^2\,\psi_\ep^{p+q}\,\d x\bigg),
   \end{split}
  \end{equation}
  where $C_0 = C_0(p,q,\lambda,\|u\|_{L^\infty(\Omega)},f) > 0$
  is a suitable constant and
  \begin{equation*}
   \begin{split}
   & I_p := \int_{\Omega_\lambda}
   |\nabla w_{\lambda}^+|^{p-1}\psi_\ep^{p+q-1}|\nabla \psi_\ep| w_{\lambda}^{+}\, \d x, \\[0.1cm]
   &  I_q := \int_{\Omega_\lambda}
   a(x)|\nabla w_{\lambda}^+|^{q-1}\psi_\ep^{p+q-1}|\nabla \psi_\ep| w_{\lambda}^{+}\, \d x.
   \end{split}
  \end{equation*}
  We now proceed to estimate both $I_p$ and $I_q$. 
  
  \vspace*{0.08cm}
  
   To begin with, we split the set $\Omega_\lambda$ as 
  $\Omega_\lambda = \Omega^{(1)}_\lambda\cup\Omega^{(2)}_\lambda$, where
  \begin{align*}
  & \Omega^{(1)}_\lambda = \{x\in\Omega_\lambda\setminus R_\lambda(\Gamma):
   \,|\nabla u_\lambda(x)|
   < 2|\nabla u|\}\qquad\text{and} \\[0.08cm]
  & \qquad \Omega^{(2)}_\lambda = \{
   x\in\Omega_\lambda\setminus R_\lambda(\Gamma):\,|\nabla u_\lambda(x)|
   \geq 2|\nabla u|\};
   \end{align*}
  accordingly, 
  since Remark \ref{rem:capacityprop}-(2) ensures that 
  $|R_\lambda(\Gamma)| = 0$,  we write 
  $$I_p = I_{p,1}+I_{p,2},\qquad\text{with
  $I_{p,i} = \int_{\Omega^{(i)}_\lambda}
   \{\cdots\}\,\d x \quad(i = 1,2)$}.$$
  We then proceed by estimating $I_{p,1},\,I_{p,2}$ separately.
  \medskip
  
  \textsc{Step I: Estimate of $I_{p,1}$}. By definition, for every
  $x\in \Omega_\lambda^{(1)}$ we have
  \begin{equation} \label{eq.estimInOmegafirst}
   |\nabla u_\lambda(x)|+|\nabla u(x)| < 3|\nabla u(x)|;
  \end{equation}
  Using the H\"{o}lder inequality with conjugate exponents $(\tfrac{p}{p-1},p)$ and \eqref{eq.estimInOmegafirst}, we get
\begin{equation} \label{eq.estimI1ptouse} 
  \begin{aligned}
   I_{p,1} & \leq \left( \int_{\Omega^{(1)}_\lambda} |\nabla w_{\lambda}^+|^p 
   \psi_\ep^{\frac{p(p+q-1)}{p-1}}\, \d x\right)^{\tfrac{p-1}{p}} 
   \left(\int_{\Omega^{(1)}_\lambda} |\nabla \psi_\ep|^p (w_{\lambda}^+)^p \, \d x\right)^{\tfrac{1}{p}}\\
   &\leq \left( \int_{\Omega^{(1)}_\lambda} \left( |\nabla u| + |\nabla u_{\lambda}|\right)^p 
   \psi_\ep^{\frac{p(p+q-1)}{p-1}}\, \d x\right)^{\tfrac{p-1}{p}} \left(\int_{\Omega^{(1)}_\lambda} |\nabla \psi_\ep|^p (w_{\lambda}^+)^p \, \d x\right)^{\tfrac{1}{p}}\\
   &\leq 3^{p-1} \left( \int_{\Omega^{(1)}_\lambda} |\nabla u|^p 
   \psi_\ep^{\frac{p(p+q-1)}{p-1}}\, \d x\right)^{\tfrac{p-1}{p}} \left(\int_{\Omega^{(1)}_\lambda} |\nabla \psi_\ep|^p (w_{\lambda}^+)^p \, \d x\right)^{\tfrac{1}{p}}\\
   &\leq C \left( \int_{\Omega_\lambda} |\nabla u|^p \, \d x\right)^{\tfrac{p-1}{p}} \left(\int_{\Omega_\lambda}\left( |\nabla \psi_\ep|^p +a(x) |\nabla \psi_\ep|^q \right) \, \d x\right)^{\tfrac{1}{p}},
  \end{aligned}
  \end{equation}
  \vspace*{0.05cm}
  where $C = C(p,\lambda,\|u\|_{L^\infty(\Omega_\lambda)}) > 0$. 
    \medskip
  
  \textsc{Step II: Estimate of $I_{p,2}$}.
  By definition, for every $x\in\Omega_\lambda^{(2)}$ we have
  \begin{equation} \label{eq.estimInOmegasec}
  \frac{1}{2}|\nabla u_\lambda| \leq 
   |\nabla u_\lambda|-|\nabla u|
   \leq |\nabla w_\lambda| \leq |\nabla u_\lambda|+|\nabla u|
   \leq \frac{3}{2}|\nabla u_\lambda|.
  \end{equation}
  Using the weighted
  Young i\-ne\-qua\-li\-ty with conjugate exponents $(\tfrac{p}{p-1},p)$ and \eqref{eq.estimInOmegasec},   
  for every $\rho > 0$ we get
  \begin{equation} \label{eq.estimI2ptouse}
  \begin{aligned}
   I_{p,2} & \leq C\rho \int_{\Omega^{(2)}_\lambda} |\nabla w_\lambda^+|^p 
   \psi_\ep^{\frac{p(p+q-1)}{p-1}}\,\d x + \dfrac{C}{\rho^{p-1}}
   \int_{\Omega^{(2)}_\lambda} |\nabla \psi_\ep|^p (w_{\lambda}^+)^p \, \d x   \\
   & \leq C\rho \int_{\Omega^{(2)}_\lambda} 
   \left( |\nabla u| + |\nabla u_{\lambda}|\right)^{p}
   \psi_\ep^{\frac{p(p+q-1)}{p-1}}\,\d x + \dfrac{C}{\rho^{p-1}}
   \int_{\Omega^{(2)}_\lambda} |\nabla \psi_\ep|^p (w_{\lambda}^+)^p \, \d x   \\
   &\leq C\rho \int_{\Omega^{(2)}_\lambda} \left( |\nabla u| + |\nabla u_{\lambda}|\right)^{p-2}  |\nabla u_{\lambda}|^{2}\psi_\ep^{\frac{p(p+q-1)}{p-1}}\,\d x + 
   \dfrac{C}{\rho^{p-1}} \int_{\Omega^{(2)}_\lambda} |\nabla \psi_\ep|^p (w_{\lambda}^+)^p \, \d x   \\
   &\leq C\rho \int_{\Omega^{(2)}_\lambda} \left( |\nabla u| + |\nabla u_{\lambda}|\right)^{p-2}  |\nabla u_{\lambda}|^{2}\psi_\ep^{\frac{p(p+q-1)}{p-1}}\,\d x + 
   \dfrac{C}{\rho^{p-1}} \int_{\Omega^{(2)}_\lambda} |\nabla \psi_\ep|^p (w_{\lambda}^+)^p \, \d x   \\
   &\leq C\rho \int_{\Omega_{\lambda}} \left( p \left( |\nabla u|+ |\nabla u_{\lambda}|\right)^{p-2}+qa(x)\left(|\nabla u|+ |\nabla u_{\lambda}|\right)^{q-2}\right)|\nabla w_{\lambda}^+|^2 \psi_\ep^{p+q}\,\d x \\
   &\qquad \qquad +\dfrac{C}{\rho^{p-1}}
   \int_{\Omega_{\lambda}} \left( |\nabla \psi_\ep|^p +a(x) |\nabla \psi_\ep|^q \right) \, \d x,
  \end{aligned}
  \end{equation}
   where $C = C(p,\lambda,\|u\|_{L^\infty(\Omega_\lambda)}) > 0$. Arguing similarly, just keeping 
  track of the term $a(x)$, we get an analogous estimate for $I_q$, which 
  reads as follows: for every $\rho >0$,
  \begin{equation}\label{eq.estimFinalIq}
  \begin{aligned}
  I_q &\leq C \left( \int_{\Omega_{\lambda}}|\nabla u|^{q}\, \d x\right)^{\tfrac{q-1}{q}} \left( \int_{\Omega_{\lambda}}   \left(|\nabla \psi_\ep|^p + a(x) |\nabla \psi_\ep|^q\right)\, \d x\right)^{\tfrac{1}{q}} \\
  &\qquad + C \rho \int_{\Omega_{\lambda}} \left( p \left( |\nabla u|+ |\nabla u_{\lambda}|\right)^{p-2}+qa(x)\left(|\nabla u|+ |\nabla u_{\lambda}|\right)^{q-2}\right)|\nabla w_{\lambda}^+|^2 \psi_\ep^{p+q}\,\d x \\
   &\qquad \qquad +\dfrac{C}{\rho^{q-1}}
   \int_{\Omega_{\lambda}} \left( |\nabla \psi_\ep|^p +a(x) |\nabla \psi_\ep|^q \right) \, \d x,
  \end{aligned}
  \end{equation}
   where $C = C(p,\lambda,\|u\|_{L^\infty(\Omega_\lambda)},\|a\|_{L^\infty(\Omega_\lambda)}) > 0$.
  \medskip
  
  Going back to \eqref{eq.tostartfrom}, and gathering \eqref{eq.estimI1ptouse},
  \eqref{eq.estimI2ptouse} and \eqref{eq.estimFinalIq}, we finally derive
  \begin{equation*}
  \begin{aligned}
   C_1 &\int_{\Omega_\lambda}\psi_\ep^{p+q}\,\big\{
   p(|\nabla u|+|\nabla u_\lambda|)^{p-2}
   + q a(x)\,(|\nabla u|+|\nabla u_\lambda|)^{q-2}\big\}
   \cdot |\nabla w_\lambda^+|^2\,\d x \\
   &\leq C\left( \int_{\Omega_{\lambda}}|\nabla u|^{p}\, \d x\right)^{\tfrac{p-1}{p}}
   \left( \int_{\Omega_{\lambda}}\left(|\nabla \psi_\ep|^p + 
   a(x) |\nabla \psi_\ep|^q\right)\, \d x\right)^{\tfrac{1}{p}} \\
   & \quad
   + C\left( \int_{\Omega_{\lambda}}|\nabla u|^{q}\, \d x\right)^{\tfrac{q-1}{q}}
   \left( \int_{\Omega_{\lambda}}\left(|\nabla \psi_\ep|^p +
    a(x) |\nabla \psi_\ep|^q\right)\, \d x\right)^{\tfrac{1}{q}} \\
   &\quad
   +C\rho \int_{\Omega_{\lambda}} \big\{p \left( |\nabla u|+ |\nabla u_{\lambda}|\right)^{p-2}+qa(x)\left(|
   \nabla u|+ |\nabla u_{\lambda}|\right)^{q-2}\big\}
   |\nabla w_{\lambda}^+|^2 \psi_\ep^{p+q}\,\d x \\
   &\quad 
   +C\bigg(\frac{1}{\rho^{p-1}}+\frac{1}{\rho^{q-1}}\bigg)\int_{\Omega_{\lambda}} 
   \left( |\nabla \psi_\ep|^p +a(x) |\nabla \psi_\ep|^q \right) \, \d x
   + C_0\int_{\Omega_\lambda}
   (w_\lambda^+)^2\,\psi_\ep^{p+q}\,\d x,
   \end{aligned}
  \end{equation*}
  \noindent for every $\rho>0$ and for a suitable positive constant $C$ only depending
  on $p,\lambda$, $\|u\|_{L^\infty(\Omega_\lambda)}$ and $\|a\|_{L^\infty(\Omega_\lambda)}$.
 We are finally ready
  to conclude the proof: in fact, choosing
  $\rho > 0$ in such a way that
  $$C_1-C \rho< \frac{1}{2},$$
  and letting $\varepsilon\to 0$ with the aid of Fatou's lemma
  (remind the properties (1)-to-(4) of the function
  $\psi_\ep$ and that
  $u\in C^1(\overline{\Omega}_\lambda)$ if $\lambda < 0$), we obtain
  $$\int_{\Omega_\lambda} \big( p(|\nabla u|+|\nabla u_\lambda|)^{p-2}
   + q a(x)\,(|\nabla u|+|\nabla u_\lambda|)^{q-2}\big)
   |\nabla w_\lambda^+|^2\,\d x
   \leq 
   C_0\!\int_{\Omega_\lambda}
   (w_\lambda^+)^2\,\d x =: \mathbf{c}_0,$$
   and $\mathbf{c}_0$ depends only on $p,q,\lambda$ and the $L^{\infty}$-norm of $u$.
   This is ends the proof.
 \end{proof}
 Another key tool for the proof of Theorem \ref{thm:symmetry} is the upcoming Lemma \ref{conncomp} which has been proved in \cite[Lemma 2.5]{BEV} {\em for every} $1<p<q<N$. To better understand this result, we first introduce
 a notation: for every fixed
 $\lambda\in(\mathbf{a},0)$, we define 
 \begin{equation*}
  \mathcal{Z}_\lambda := \big\{x\in\Omega_\lambda\setminus R_\lambda(\Gamma):\,
  \nabla u(x) = \nabla u_\lambda(x) = 0\big\}.
 \end{equation*}
 We also notice that, 
 since $u,u_\lambda\in C^1(\overline{\Omega}_\lambda\setminus R_\lambda(\Gamma))$,
 the set $\mathcal{Z}_\lambda$ is closed (in $\Omega_\lambda$).
 \begin{lem}\label{conncomp}
  Let $\lambda\in (\mathbf{a},0)$ and let $\mathcal{C}_\lambda \subseteq \Omega_\lambda \setminus
	(R_\lambda(\Gamma) \cup \mathcal{Z}_\lambda)$ be a
	\emph{connected component} of \emph{(}the open set\emph{)}
	$\Omega_\lambda \setminus
	(R_\lambda(\Gamma) \cup \mathcal{Z}_\lambda)$. 	
	If 
	$u \equiv u_\lambda$ in $\mathcal{C}_\lambda$, then 
	$$\mathcal{C}_\lambda = \emptyset.$$
 \end{lem}
 Finally, we will use the following key technical result.
 \begin{lem}\label{lem:SobolevSplitting}
	Let $1<p<q<2$ be fixed, and
	let $\lambda\in (\mathbf{a},0)$  Then, it is possible 
	to find a constant $\mathbf{c}  > 0$ such that the following estimate holds:
	\begin{equation*}
		\int_{\Omega_\lambda} |\nabla w_{\lambda}^+|^p\, dx \leq  \mathbf{c} \ \left(\int_{\Omega_\lambda} (|\nabla u| + |\nabla u_{\lambda}|)^{p-2}| \nabla w_{\lambda}^+|^2 \, dx\right)^{\frac{p}{2}}.
	\end{equation*}
	Here, $u_\lambda$ and $w_\lambda$ are as in
	\eqref{eq.defulambda} and \eqref{eq:defwlambda}, respectively.
\end{lem}
 Taking into account the previous
 Lemma \ref{leaiuto}, 
 the proof of Lemma \ref{lem:SobolevSplitting} is totally analogous to that
 of \cite[Lemma 2.4]{BEMV}, and therefore we skip it.
\section{Proof of Theorem \ref{thm:symmetry}} \label{sec.mainresult1}

\begin{proof}[Proof of Theorem \ref{thm:symmetry}] 
By assumptions, the singular set $\Gamma$ is contained in the hyperplane $\{x_1 = 0\}$, then the moving plane procedure can be started in the standard way, see e.g \cite{EMS} for the $p$-laplacian case, by using the weak comparison principle in small domains, see \cite[Theorem 4.3]{riey}.
	Indeed, for $\mathbf{a}
	< \lambda < \mathbf{a} + \tau$ with $\tau>0$ small enough, the singularity does not play any role. Therefore, recalling that $w_\lambda$ has a
	singularity at $\Gamma$ and at $R_\lambda (\Gamma)$, we have that
	$w_\lambda \leq 0$ in $\Omega_\lambda$. To proceed
	further we define
	\begin{equation}\nonumber
	\Lambda_0=\{\mathbf{a}<\lambda<0 : u\leq
	u_{t}\,\,\,\text{in}\,\,\,\Omega_t\setminus
	R_t(\Gamma)\,\,\,\text{for all $t\in(\mathbf{a},\lambda]$}\}
	\end{equation}
	and $\lambda_0 = \sup \Lambda_0$, since we proved above that $\Lambda_0$ is not empty. To prove our result we have to
	show that $\lambda_0 = 0$.
	To do this we
	assume that $\lambda_0 < 0$ and we reach a contradiction by proving
	that $u \leq u_{\lambda_0 + \tau}$ in $\Omega_{\lambda_0 + \tau}
	\setminus R_{\lambda_0 + \tau} (\Gamma)$ for any $0 < \tau <
	\bar{\tau}$ for some small $\bar{\tau}>0$. We  remark that
	$|\mathcal{Z}_{\lambda_0}|=0$, see \cite{DS1, riey}.
	Let us take $\mathcal{H}_{\lambda_0}\subset \Omega_{\lambda_0}$ be an open set such that 
	$$\mathcal Z_{\lambda_0}\cap\Omega_{\lambda_0}\subset
	\mathcal{H}_{\lambda_0} \subset \subset \Omega.$$ 
	We note that the existence of such
	 a set is guaranteed by the Hopf lemma, see, e.g., \cite[Theorem A.2]{BEV}. 
	 Moreover note  that, since $|\mathcal Z_{\lambda_0}|=0$, we can take $\mathcal{H}_{\lambda_0}$ of arbitrarily small measure. By
	continuity we know that $u \leq u_{\lambda_0}$ in $\Omega_{\lambda_0}
	\setminus R_{\lambda_0} (\Gamma)$.
	We can exploit the classical strong comparison principle 
	(see \cite[Theorem 1]{Serrin70} and \cite[Theorem A.1]{BEV})
	to get that, in any connected component of 
	$\Omega_{\lambda_{0}}\setminus \mathcal Z_{\lambda_0}$, we have
	$$
	u<u_{\lambda_0} \qquad\text{or}\qquad u\equiv u_{\lambda_0}.$$
	The case $u\equiv u_{\lambda_0}$ in some  connected component
	$\mathcal{C}_{\lambda_{0}}$ of $\Omega_{\lambda_{0}}\setminus \mathcal Z_{\lambda_0}$ is not
	possible, since by symmetry, it would imply the existence of  a local symmetry phenomenon and consequently that $\Omega \setminus \mathcal  Z_{\lambda_0}$ would be not connected,  in  spite of what we proved in Lemma \ref{conncomp}. Hence we deduce that $u <
	u_{\lambda_0}$ in $\Omega_{\lambda_0} \setminus R_{\lambda_0}
	(\Gamma)$. Therefore, given a compact set $\mathcal{K} \subset
	\Omega_{\lambda_0} \setminus (R_{\lambda_0} (\Gamma)\cup \mathcal{H}_{\lambda_0})$, by
	uniform continuity we can ensure that $u < u_{\lambda_0+\tau}$ in
	$\mathcal{K}$ for any $0 < \tau < \bar{\tau}$ for some small $\bar{\tau}>0$.
	Note that to do this we implicitly assume, with no loss of
	generality, that $R_{\lambda_0} (\Gamma)$ remains bounded away
	from $\mathcal{K}$. 
	
	Arguing in a similar fashion as in Lemma \ref{leaiuto}, we
	consider
	\begin{equation*}
	\varphi_\varepsilon := w^+_{\lambda_0 + \tau} 
	\psi_\varepsilon^{p+q}\cdot\mathbf{1}_{\Omega_{\lambda_0 + \tau}}
	= \begin{cases}
	w^+_{\lambda_0 + \tau} 
	\psi_\varepsilon^{p+q}, & \text{in $\Omega_{\lambda_0+\tau}$}, \\
	0, & \text{otherwise}.
	\end{cases}
	\end{equation*}
	By density arguments as above, we plug $\varphi_\varepsilon$ as test
	function in \eqref{debil1} and \eqref{eq.PDEulambda} so that,
	subtracting, we get
	
	\begin{equation} \label{eq:Step1}
	\begin{split} &\int_{\Omega_{\lambda_0 + \tau} \setminus \mathcal{K}} \langle|\nabla u|^{p-2} \nabla u - |\nabla u_{\lambda_0 + \tau}|^{p-2}\nabla u_{\lambda_0 + \tau},
	\nabla w^+_{\lambda_0 + \tau}\rangle \, \psi_\varepsilon^{p+q} \, \d x\\
	&+\int_{\Omega_{\lambda_0 + \tau} \setminus \mathcal{K}} \langle|\nabla u|^{q-2} \nabla u
	- |\nabla u_{\lambda_0 + \tau}|^{q-2} \nabla u_{\lambda_0 + \tau},
	\nabla w^+_{\lambda_0 + \tau}\rangle \, \psi_\varepsilon^{p+q} \, \d x\\
	&+ (p+q) \int_{\Omega_{\lambda_0 + \tau} \setminus \mathcal{K}} \langle|\nabla u|^{p-2}
	\nabla u - |\nabla u_{\lambda_0 + \tau}|^{p-2} \nabla u_{\lambda_0 +
		\tau}, \nabla \psi_\varepsilon\rangle \, \psi_\varepsilon^{p+q-1} w_{\lambda_0 + \tau}^+ \, \d x \\
	&+ (p+q) \int_{\Omega_{\lambda_0 + \tau} \setminus \mathcal{K}} \langle|\nabla u|^{q-2}
	\nabla u - |\nabla u_{\lambda_0 + \tau}|^{q-2} \nabla u_{\lambda_0 +
		\tau}, \nabla \psi_\varepsilon\rangle \,	\psi_\varepsilon^{p+q-1} w_{\lambda_0 + \tau}^+ \, \d x \\
	&= \int_{\Omega_{\lambda_0 + \tau} \setminus \mathcal{K}} (f(u)-f(u_\lambda))
	w_{\lambda_0 + \tau}^+ \psi_\varepsilon^{p+q} \, \d x.
	\end{split}
	\end{equation}
	Now we split the set $\Omega_{\lambda_0 + \tau}
	\setminus \mathcal{K}$ as the union of two disjoint subsets $\Omega^{(1)}_{\lambda_0 +
	\tau}$ and $\Omega^{(2)}_{\lambda_0 + \tau}$ such that $\Omega_{\lambda_0 +\tau} \setminus \mathcal{K}= \Omega^{(1)}_{\lambda_0 + \tau} \cup \Omega^{(2)}_{\lambda_0 +	\tau}$.
	In particular, we set
	\begin{equation}\nonumber
	\begin{split}
	 \Omega^{(1)}_{\lambda_0 +	\tau} &:= \{ x \in \Omega_{\lambda_0 + \tau} \setminus
	\mathcal{K} \ : \ |
	\nabla u_{\lambda_0 + \tau} (x)| < 2| \nabla u (x)|  \}\quad
	\text{and}\\ \\
	\Omega^{(2)}_{\lambda_0 +	\tau} &:= \{ x \in \Omega_{\lambda_0 + \tau} \setminus
	\mathcal{K} \ : \ | \nabla u_{\lambda_0 + \tau} (x)| \geq  2| \nabla u
	(x)|\}.
	\end{split}
	\end{equation}
	From \eqref{eq:Step1} and using \eqref{eq:inequalities}, repeating verbatim arguments along the proof of Lemma \ref{leaiuto}, we get
	
\begin{equation}\label{eq:final}
\begin{split}
\int_{\Omega_{\lambda_0+\tau} \setminus \mathcal{K}} &\left( p(|\nabla u|+|\nabla u_{\lambda_0+\tau}|)^{p-2} + q a(x)\,(|\nabla u|+|\nabla u_{\lambda_0+\tau}|)^{q-2}\right)|\nabla w_{\lambda_0+\tau}^+|^2 \, \d x  \\[0.1cm]
\leq& C_0 \int_{\Omega_{\lambda_0 + \tau} \setminus \mathcal{K}} (w_{\lambda_0	+ \tau}^+)^2 \, \d x,
\end{split}
\end{equation}
where $C_0=C_0(p,q,\lambda_0, \tau, \|u\|_{L^\infty(\Omega)},f)$. Clearly, the left hand side can estimate from below as follows
\begin{equation}\label{eq:LHS}
\begin{split}
&\int_{\Omega_{\lambda_0+\tau} \setminus \mathcal{K}}(|\nabla u|+|\nabla u_{\lambda_0+\tau}|)^{p-2}|\nabla w_{\lambda_0+\tau}^+|^2 \, \d x \\
&\leq \int_{\Omega_{\lambda_0+\tau} \setminus \mathcal{K}}\left( p(|\nabla u|+|\nabla u_{\lambda_0+\tau}|)^{p-2} + q a(x)\,(|\nabla u|+|\nabla u_{\lambda_0+\tau}|)^{q-2}\right)|\nabla w_{\lambda_0+\tau}^+|^2 \, \d x.
\end{split}
\end{equation}
For the right hand side, we apply H\"{o}lder's inequality with  exponents $\left( \tfrac{p^{\ast}-2}{p^{\ast}}, \tfrac{p^{\ast}}{2}\right)$,
the clas\-sical Sobolev inequality and Lemma \ref{lem:SobolevSplitting}: this gives
\begin{equation}\label{eq:RHS}
\begin{split}
C_0 &\int_{\Omega_{\lambda_0 + \tau} \setminus \mathcal{K}} (w_{\lambda_0	+ \tau}^+)^2 \, \d x \leq C |\Omega_{\lambda_0 + \tau} \setminus \mathcal{K}|^{\tfrac{p^{\ast}-2}{p^{\ast}}} \left( \int_{\Omega_{\lambda_0 + \tau} \setminus \mathcal{K}}|w_{\lambda_0 + \tau}^+|^{p^{\ast}}\, \d x \right)^{\tfrac{2}{p^{\ast}}}\\
&\leq C |\Omega_{\lambda_0 + \tau} \setminus \mathcal{K}|^{\tfrac{p^{\ast}-2}{p^{\ast}}} \left( \int_{\Omega_{\lambda_0 + \tau} \setminus \mathcal{K}}|\nabla w_{\lambda_0 + \tau}^+|^{p}\, \d x \right)^{\tfrac{2}{p}}\\
&\leq C |\Omega_{\lambda_0 + \tau} \setminus \mathcal{K}|^{\tfrac{p^{\ast}-2}{p^{\ast}}} \int_{\Omega_{\lambda_0 + \tau} \setminus \mathcal{K}} \left( |\nabla u|+ |\nabla u_{\lambda}|\right)^{p-2}|\nabla w_{\lambda}^+|^2 \, \d x.
\end{split}
\end{equation}
Gathering together \eqref{eq:final}, \eqref{eq:LHS} and \eqref{eq:RHS}, we get
\begin{equation} \label{eq:estimkey}
\begin{split}
&\int_{\Omega_{\lambda_0+\tau} \setminus \mathcal{K}}(|\nabla u|+|\nabla u_{\lambda_0+\tau}|)^{p-2}|\nabla w_{\lambda_0+\tau}^+|^2 \, \d x \\
&\qquad \leq C |\Omega_{\lambda_0 + \tau} \setminus \mathcal{K}|^{\tfrac{p^{\ast}-2}{p^{\ast}}} \int_{\Omega_{\lambda_0 + \tau} \setminus \mathcal{K}} \left( |\nabla u|+ |\nabla u_{\lambda}|\right)^{p-2}|\nabla w_{\lambda}^+|^2 \, \d x.
\end{split}
\end{equation}
For $\bar{\tau}$
	small and $\mathcal{K}$ large, we may assume that
	$$C |\Omega_{\lambda_0 + \tau} \setminus \mathcal{K}|^{\tfrac{p^{\ast}-2}{p^{\ast}}} < 1.$$
	We can then deduce that
	\begin{equation*}
	\int_{\Omega_{\lambda_0 + \tau} \setminus \mathcal{K}}(w_{\lambda}^+)^2 \, \d x = \int_{\Omega_{\lambda_0 + \tau}}(w_{\lambda}^+)^2 \, \d x = 0
	\end{equation*}
	proving that $u \leq u_{\lambda_0+\tau}$ in $\Omega_{\lambda_0 +
		\tau} \setminus R_{\lambda_0 + \tau} (\Gamma)$ for any $0 < \tau <
	\bar{\tau}$ for some small $\bar{\tau}>0$. Such a contradiction
	shows that
	$$ \lambda_0 = 0.$$
	Since the moving plane procedure can be performed in the same way
	but in the opposite direction, then this proves the desired symmetry
	result. The fact that the solution is increasing in the
	$x_1$-direction in $\{x_1 < 0\}$ is implicit in the moving plane
	procedure.	
  \end{proof}
  \begin{rem} \label{rem:casogeneraleDett}
  By carefully scrutinizing the proof of Theorem \ref{thm:symmetry},
  one can easily see that the key point in all the argument is
  estimate \eqref{eq:estimkey}, \emph{which only depends on $p$}.
  For this reason, we can prove Theorem \ref{thm:symmetry} in the general
  case
  $$\frac{2N}{N+2}<p\leq q$$
  by considering the following facts.
  \begin{itemize}
   \item[(a)] 
   When $\frac{2N}{N+2}<p<2$ and $q\geq 2$ one can argue \emph{exactly as we did in this paper},
   up to modifying the proof of Lemma \ref{leaiuto} to cover the case $q\geq 2$;
   this can be done by adapting the argument exploited
   in the proof
   of \cite[Lemma 2.4]{BEV}.
   \medskip
   
   \item[(b)] When $q\geq p\geq 2$, one can proceed exactly as in
   \cite{BEV}, up to slightly modifying the proof of \cite[Lemma 2.4]{BEV}
   to take into account the weaker assumption \eqref{eq:assCappqzero};
   this can be done by using the same approach exploited
   in the proof of Lemma \ref{leaiuto}.
  \end{itemize}
  \end{rem}

%

%
%
%
%
%

\end{document}